\documentclass[11pt]{article}
\usepackage{amsmath,amssymb,verbatim}
\usepackage{times}
\usepackage{amsfonts}
\usepackage[dvips]{graphicx}
\usepackage{amsfonts}
\usepackage{amsthm}
\usepackage{epsfig}
\usepackage{graphics}
\usepackage{color}
\usepackage[utf8]{inputenc}
\usepackage{tikz}
\usepackage{amsthm}
\usepackage{setspace}
\usepackage{paralist}
\usepackage{amssymb}
\usepackage{caption}
\usepackage{subcaption}
\usepackage{amsmath,amssymb,verbatim}
\usepackage{times}
\usepackage{amsfonts}
\usepackage{graphicx}
\usepackage{amsfonts}
\usepackage{amsthm}
\usepackage{epsfig}
\usepackage{graphics}
\usepackage{color}
\usepackage[utf8]{inputenc}
\usepackage{amsthm}
\usepackage{setspace}
\usepackage{algorithm}
\usepackage{algorithmic}
\usepackage{setspace}
\usepackage{filecontents}
\usepackage{algorithm}
\usepackage{algorithmic}

\theoremstyle{plain}
\newtheorem{theorem}{Theorem}[section]
\newtheorem{lem}[theorem]{Lemma}
\newtheorem{prop}[theorem]{Proposition}
\newtheorem{col}[theorem]{Corollary}
\theoremstyle{definition}
\newtheorem{defn}[theorem]{Definition}

\newtheorem{remark}[theorem]{Remark}

\newtheorem{example}[theorem]{Example}

\newcommand{\A}{{\mathcal{A}}}

\newcommand{\II}{{\mathcal{I}}}

\newcommand{\FF}{\mbox{Facets}}
\newcommand{\VV}{\mbox{V}}
\newcommand{\FFF}{{\mathcal{F}}}
\newcommand{\CC}{{\mathcal{C}}}

\newcommand{\KK}{{\mathbb{K}}} 
\newcommand{\lcm}{\operatorname{lcm}}

\newcommand{\Supp}{\ Supp \ }

\numberwithin{equation}{section}
\title{When is a Squarefree Monomial Ideal of Linear Type ?}
\author{Ali Alilooee \and Sara Faridi}

\begin{document}
 
\maketitle

\begin{abstract}
In $1995$ Villarreal gave a combinatorial description of the equations of Rees algebras of quadratic squarefree monomial ideals. 
His description was based on the concept of closed even walks in a graph. In this paper we will generalize his results to all squarefree monomial ideals
by defining even walks in a simplicial complex. We show that simplicial complexes with no even walks have facet ideals that are of linear type, generalizing
Villarreal's work.         
\end{abstract}


\section{Introduction}
Rees algebras are of special interest in algebraic geometry and commutative algebra since they describe 
the blowing up of the spectrum of a ring along the subscheme defined by an ideal.
The Rees algebra of an ideal can also be viewed as a quotient of a polynomial ring. If $I$ is an ideal of a ring $R$,
we denote the Rees algebra of $I$ by $R[It]$, and we can represent  $R[It]$ as $S/J$ where $S$ is a polynomial
ring over $R$. The ideal $J$ is called the {\bf defining ideal} of $R[It]$. Finding generators of $J$ is difficult and crucial for a better understanding of $R[It]$. Many authors
have worked to get a better insight into these generators in special classes of ideals, such as those with special height, special embedding dimension and so on.  

When $I$ is a monomial ideal, using methods from Taylor's thesis~\cite{Taylor1966} one can describe the generators of $J$ as binomials. 
Using this fact, Villarreal \cite{Villarreal1995} gave a combinatorial 
characterization of $J$ in the case of degree $2$ squarefree monomial ideals. His work led Fouli and Lin~\cite{fouli2013} to consider the question of characterizing generators of $J$ when $I$ is a 
squarefree monomial ideal in any degree. With this purpose in mind we define simplicial even walks, and show that for all squarefree monomial ideals, 
they identify generators of $J$ that may be obstructions to $I$ being of linear type. We show that in dimension $1$, simplicial even walks are the same as closed even walks of graphs.
  We then further investigate properties of simplicial even walks, and reduce the problem of checking whether 
an ideal is of linear type to identifying simplicial even walks. At the end of the paper we give a new proof for Villarreal's Theorem (Corollary~\ref{col:vill}).     


\section{Rees algebras and their equations}


Let $I$ be a monomial ideal in a polynomial ring $R=\KK[x_1,\dots,x_n]$ over a field $\KK$. We denote the {\bf Rees algebra} of $I=(f_1,\dots,f_{q})$ by $R[It]=R[f_1t,\dots,f_qt]$ and 
consider the homomorphism $\psi$ of algebras 
\begin{align*}
 \psi:R[T_1,\dots,T_q]\longrightarrow R[It], \hspace{.06 in}  T_i\mapsto f_{i}t.
\end{align*}
If $J$ is the kernel of $\psi$, we can consider the Rees algebra $R[It]$ as the quotient of the polynomial ring $R[T_1,\dots,T_q]$. 
The ideal $J$ is called the \textbf{defining ideal} of $R[It]$ and 
its minimal generators are called the \textbf{Rees equations} of $I$.
These equations carry a lot of information about $R[It]$; see for example ~\cite{Vasconcelos1994} for more details. 


\begin{defn}
For integers $s,q\geq 1$ we define 
$$\II_{s}=\{(i_1,\dots,i_s):1\leq i_1\leq i_2\leq\dots\leq i_{s}\leq q\}\subset {\mathbb{N}}^{s}.$$ 
Let $\alpha=(i_1,\dots,i_s)\in\II_{s}$ and $f_1,\dots,f_q$ be monomials in $R$ and 
$T_1,\dots,T_q$ be variables. We use the following notation for the rest of this paper. If $t\in\{1,\dots,s\}$     
\begin{itemize}
\item $\Supp(\alpha)=\{i_1,\dots,i_s\}$; 
\item $\widehat{\alpha}_{i_t}=(i_1,\dots,\widehat{i}_{t},\dots,i_s)$;
\item ${T}_{\alpha}=T_{i_1}\dots T_{i_s}$ and $\Supp(T_{\alpha})=\{T_{i_1},\dots ,T_{i_s}\}$;
\item ${f}_{\alpha}=f_{i_1}\dots f_{i_s}$;
\item $\displaystyle \widehat{f}_{\alpha_t}=f_{i_1}\dots\widehat{f}_{i_t}\dots f_{i_s}=\frac{f_{\alpha}}{f_{i_t}}$;
\item $\displaystyle \widehat{T}_{\alpha_t}=T_{i_1}\dots\widehat{T}_{i_t}\dots T_{i_s}=\frac{T_{\alpha}}{T_{i_t}}$;
\item $\alpha_{t}(j)=(i_1,\dots,i_{t-1},j,i_{t+1},\dots,i_s)$, for $j\in\{1,2,\dots,q\}$ and $s\geq 2$.
\end{itemize}
\end{defn}
For an ideal $I=(f_1,\dots,f_q)$ of $R$ the defining ideal $J$ of $R[It]$ is graded and
$$J=J^{\prime}_1\oplus J^{\prime}_2\oplus\cdots$$
where $J^{\prime}_s$ for $s\geq 1$ is the $R$-module.

The ideal $I$ is said to be \textbf{of linear type} if $J=(J^{\prime}_1)$; in other words, the defining ideal of $R[It]$ is generated by linear forms in the variables $T_1,\dots,T_q$.

\begin{defn}
Let $I=(f_1,\dots,f_q)$ be a monomial ideal, $s\geq 2$ and $\alpha,\beta\in\II_{s}$. We define 
\begin{eqnarray}\label{eqn:lineartype}
T_{\alpha,\beta}(I)=\left(\frac{\lcm(f_{\alpha},f_{\beta})}{f_{\alpha}}\right)T_{\alpha}-
\left(\frac{\lcm(f_{\alpha},f_{\beta})}{f_{\beta}}\right)T_{\beta}.\label{equation:new}
\end{eqnarray}
When $I$ is clear from the context we use $T_{\alpha,\beta}$ to denote $T_{\alpha,\beta}(I)$.   
\end{defn}
\begin{prop}(D. Taylor~\cite{Taylor1966})\label{eqn:taylortheorem}
Let $I=(f_1,\dots,f_q)$ be a monomial ideal in $R$ and $J$ be the defining ideal of $R[It]$. Then for $s\geq 2$ we have 
\begin{align*}\label{eqn:2}
J^{\prime}_s= \left\langle T_{\alpha,\beta}(I): \alpha,\beta\in \II_{s}
\right\rangle.   
\end{align*}
Moreover, if $m=\gcd{(f_1,\dots,f_{q})}$ and $I^{\prime}=(f_1/m,\dots,f_{q}/m)$, then for every $\alpha,\beta\in \II_{s}$ we have 
$$T_{\alpha,\beta}(I)=T_{\alpha,\beta}(I^{\prime})$$
and hence $R[It]=R[I^{\prime}t]$.
\end{prop}
In light of Proposition~\ref{eqn:taylortheorem}, we will always assume that if $I=(f_1,\dots,f_q)$ then 
\begin{eqnarray*}
\gcd{(f_1,\dots,f_q)}=1.
\end{eqnarray*}
We will also assume $\Supp(\alpha)\cap\Supp(\beta)=\emptyset$, since otherwise $T_{\alpha,\beta}$ reduces to those with this property.
This is because if $t\in \Supp(\alpha)\cap\Supp(\beta)$ then we have 
$T_{\alpha,\beta}=T_{t}T_{\widehat{\alpha}_{t},\widehat{\beta}_{t}}.$  

For this reason we define 
\begin{eqnarray}\label{eqn:2}
\displaystyle J_s= \left\langle T_{\alpha,\beta}(I): \alpha,\beta\in \II_{s},\Supp(\alpha)\cap\Supp(\beta)=\emptyset
\right\rangle   
\end{eqnarray}
as an $R$-module. Clearly $J=J_1S+J_2S+\cdots$.
\begin{defn}\label{den:def}
Let $I=(f_1,\dots,f_q)$ be a squarefree monomial ideal in $R$ and $J$ be the defining ideal of $R[IT]$, $s\geq 2$, and $\alpha=(i_1,\dots,i_s),\beta=(j_1,\dots,j_s)\in \II_s$. 
We call $T_{\alpha,\beta}$ {\bf redundant} if it is a redundant generator of $J$, coming from lower degree; i.e. 
$$\displaystyle T_{\alpha,\beta}\in J_1S+\dots+ J_{s-1}S.$$
 
\end{defn}

\section{Simplicial even walks}

By using the concept of closed even walks in graph theory Villarreal~\cite{Villarreal1995} classified all Rees equations of  
edge ideals of graphs in terms of  
closed even walks. In this section our goal is to define an even walk in a simplicial complex in order to classify all irredundant Rees equations of squarefree monomial ideals.
Motivated by the works of S. Petrović and D. Stasi in \cite{sonia2012} we generalize closed even walks from graphs to simplicial complexes. 

We begin with basic definitions that we will need later. 
\begin{defn}
A \textbf{simplicial complex} on vertex set $\VV=\left\{x_1,\dots,x_n\right\}$ is a collection $\Delta$ of subsets of $\VV$ satisfying
\begin{enumerate}

\item $\left\{x_i\right\}\in \Delta$ for all $i$,

\item $F \in \Delta , G\subseteq F \Longrightarrow G \in \Delta$.

\end{enumerate} 
The set $\VV$ is called the {\bf vertex set} of $\Delta$ and we denote it by $\VV(\Delta)$. 
The elements of $\Delta$ are called \textbf{faces} of $\Delta$ and the maximal faces under inclusion are called \textbf{facets}. We denote the simplicial complex $\Delta$
with facets $F_1,\dots,F_s$ by $\left\langle F_1,\dots,F_s\right\rangle$. We denote the set of facets of $\Delta$ with $\FF\left(\Delta\right)$.
A \textbf{subcollection} of a simplicial complex $\Delta$ is a simplicial complex whose facet set is a subset  of the facet set of $\Delta$.
\begin{defn}
Let $\Delta$ be a simplicial complex with at least three facets, ordered as $F_1,\dots,F_q$. Suppose $\bigcap F_i=\emptyset$. 
With respect to this order $\Delta$ is a
\renewcommand{\labelenumi}{(\roman{enumi})}
\begin{enumerate}
\item {\bf extended trail} if we have 
\begin{align*}
F_{i}\cap F_{i+1}\neq \emptyset\hspace{.3 in}\mbox{$i=1,\dots,q$}\hspace{.2 in}\mbox{mod $q$};
\end{align*}
\item {\bf special cycle}~\cite{herzog2008} if $\Delta$ is an extended trail in which we have 
\begin{eqnarray*}
 F_{i}\cap F_{i+1}\not\subset \bigcup_{j\notin\{i,i+1\}} F_j&\mbox{$i=1,\dots,q$}\hspace{.2 in}\mbox{mod $q$};
\end{eqnarray*}
\item {\bf simplicial cycle}~\cite{Faridi2007} if $\Delta$ is an extended trail in which we have  
\begin{eqnarray*}
F_{i}\cap F_{j}\neq \emptyset\Leftrightarrow j\in\{i+1,i-1\}&\mbox{$i=1,\dots,q$}\hspace{.2 in}\mbox{mod $q$}.
\end{eqnarray*}
\end{enumerate}
\end{defn}
We say that $\Delta$ is an extended trail (or special or simplicial cycle) if there is an order on the facets of $\Delta$ such that the specified conditions hold on that order. 
Note that 
\begin{align*}
\{\mbox{Simplicial Cycles}\}\subseteq \{\mbox{Special Cycles}\}\subseteq \{\mbox{Extended Trails}\}.
\end{align*}
\begin{defn}[Simplicial Trees and Simplicial Forests \cite{Faridi2007}\& \cite{Faridi2002}]
A simplicial complex $\Delta$ is called a {\bf simplicial forest} if $\Delta$ contains no simplicial cycle.
 If $\Delta$ is also connected, it is called a {\bf simplicial tree}.   
\end{defn}
\begin{defn}[\cite{Zheng2004}, Lemma 3.10]\label{Zheng:1}
 Let $\Delta$ be a simplicial complex. The facet $F$ of $\Delta$ is called a {\bf good leaf} of $\Delta$ if 
 the set $\left\{H\cap F; H\in \FF(\Delta)\right\}$ is totally ordered by inclusion.  
\end{defn}
Good leaves were first introduced by X. Zheng in her PhD thesis~\cite{Zheng2004} and later in~\cite{Faridi2007}. The existence of a good leaf in every tree was proved in 
\cite{herzog2008} in 2008. 
\begin{theorem}[\cite{herzog2008}, Corollary 3.4]~\label{theorem:herzog}
 Every simplicial forest contains a good leaf. 
\end{theorem}
Let $I=(f_1,\dots,f_q)$ be a squarefree monomial ideal in $R=\mathbb{K}[x_1,\dots,x_n]$.  The \textbf{facet complex} $\FFF(I)$ associated to $I$ is a simplicial complex with  
facets $F_1,\dots,F_s$, where for each $i$,
$$F_i=\left\{x_j:\hspace{.02 in} x_j | f_i,\hspace{.04 in} 1\leq j\leq n\right\}.$$
The \textbf{facet ideal} of a simplicial complex $\Delta$ is the ideal generated by the products of the variables labeling the vertices of each facet of $\Delta$; in other words 
$$\FFF(\Delta)=\left( \prod_{x \in F}x : \ F \mbox{ is a facet of $\Delta$}  \right).$$
\end{defn}

\begin{defn}[\bf{Degree}]\label{egn:defn}
Let $\Delta=\langle F_1,\dots, F_q\rangle$ be a simplicial complex, $\FFF(\Delta)=(f_1,\dots,f_{q})$ be its facet ideal and $\alpha=(i_1,\dots,i_s)\in \II_{s}$, $s\geq 1$.   
We define the  $\alpha$-{\bf degree for a vertex}  $x$ of $\Delta$ to be  
\begin{eqnarray*}
 deg_{\alpha}(x)&=\max\{m:x^{m}|f_{\alpha}\}
\end{eqnarray*}
\end{defn}
  
\begin{example}\label{eqn:first}
Consider Figure [\ref{figure2}] where  
\begin{eqnarray*}
&F_1=\{x_4,x_7,a_3\},F_2=\{x_4,x_5,a_1\},F_3=\{x_5,x_6,a_2\},\\
&F_4=\{x_2,x_3,a_2\},F_5=\{x_1,x_2,a_1\},F_6=\{x_6,x_7,a_1\}. 
\end{eqnarray*}
If we consider $\alpha=(1,3,5)$ and $\beta=(2,4,6)$  then $deg_{\alpha}(a_1)=1$ and $deg_{\beta}(a_1)=2$.
\begin{figure}
  \centering
  \begin{subfigure}[b]{0.3\textwidth}
   \centering
  \includegraphics[width=\textwidth]{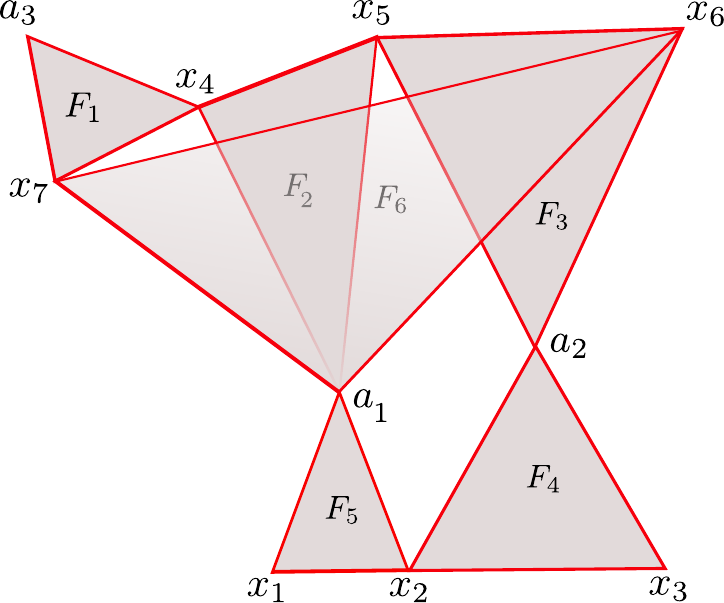} 
   \caption{Even walk}\label{figure2}  
   \end{subfigure}
  \begin{subfigure}[b]{0.4\textwidth}
   \centering
\includegraphics[width=\textwidth]{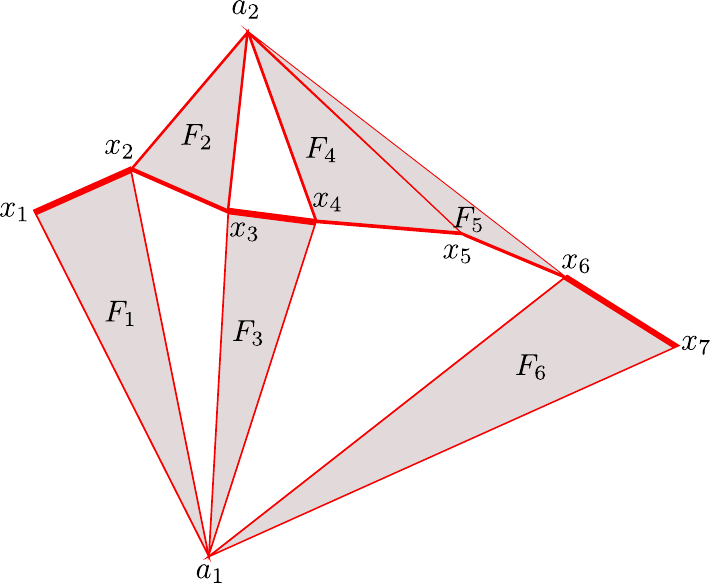}
   \caption{Not an even walk}\label{figure3}
    \end{subfigure}
    \caption{}
\end{figure}

\end{example}
Suppose $I=(f_1,\dots,f_q)$ is a squarefree monomial ideal in $R$ with $\Delta=\left\langle F_1,\dots,F_q\right\rangle $ 
its facet complex and let $\alpha,\beta\in \II_{s}$ where $s\geq 2$ is an integer.  
We set $\alpha=(i_1,\dots,i_s)$ and $\beta=(j_1,\dots,j_s)$ and consider the following sequence of not necessarily distinct facets of $\Delta$
$$\CC_{\alpha,\beta}=F_{i_1},F_{j_1},\dots,F_{i_s},F_{j_s}.$$  
Then (\ref{equation:new}) becomes  
\begin{eqnarray}\label{eqn:lcmequation}
\displaystyle T_{\alpha,\beta}(I)=\left(\prod_{deg_{\alpha}(x) < deg_{\beta}(x)} x^{deg_{\beta}(x)-deg_{\alpha}(x)}\right)T_{\alpha}-
\left(\prod_{deg_{\alpha}(x) > deg_{\beta}(x)} x^{deg_{\alpha}(x) - deg_{\beta}(x)}\right)T_{\beta}
\end{eqnarray}
where the products vary over the vertices $x$ of $\CC_{\alpha,\beta}$. 

\begin{defn}[\bf{Simplicial even walk}]\label{def:scew}
Let $\Delta=\langle F_1,\dots, F_q\rangle$ be a simplicial complex and let $\alpha=(i_1,\dots,i_s),\beta=(j_1,\dots,j_s)\in \II_{s}$,  where $s\geq 2$.
The following sequence of not necessarily distinct facets of $\Delta$  
\begin{eqnarray*}
\CC_{\alpha,\beta}= F_{i_1},F_{j_1},\dots,F_{i_s},F_{j_s}   
\end{eqnarray*}
is called a {\bf simplicial even walk}, or simply ``even walk``, if the following conditions hold
\begin{itemize}
\item  For every $i\in \Supp(\alpha)$ and $j\in \Supp(\beta)$ we have  
\begin{eqnarray*}
F_{i}\backslash F_{j}\not \subset \{x\in \VV(\Delta):   deg_{\alpha}(x) > deg_{\beta}(x)\}&\mbox{and}&
F_{j}\backslash F_{i}\not \subset \{x\in \VV(\Delta):    deg_{\alpha}(x) < deg_{\beta}(x)\}.
\end{eqnarray*}

\end{itemize}
 If $\CC_{\alpha,\beta}$ is connected, we call the even walk $\CC_{\alpha,\beta}$ a {\bf connected} even walk.
\end{defn}
\begin{remark}
It follows from the definition, if $\CC_{\alpha,\beta}$ is an even walk then $\Supp(\alpha)\cap \Supp(\beta)=\emptyset$. 
\end{remark}

\begin{example}
In Figures [\ref{figure2}] and  [\ref{figure3}] by setting $\alpha=(1,3,5),\beta=(2,4,6)$ we have $\CC_{\alpha,\beta}=F_1,\dots,F_6$ is an even walk in [\ref{figure2}] but
in [\ref{figure3}] $\CC_{\alpha,\beta}=F_1,\dots,F_6$ is not an even walk because 
$$F_1\backslash F_2=\{x_1,a_1\}=\{x:deg_{\alpha}(x)>deg_{\beta}(x)\}.$$

\end{example}
\begin{remark}
A question which naturally arises here is if a minimal even walk (an even walk that does not properly contain another even walk) can have repeated facets. The answer is positive since 
for instance, the bicycle graph in Figure~\ref{fig:bicycle} is a minimal even walk, because of Theorem~\ref{them:graphwalk} below, but it has a pair of repeated edges.      
\begin{figure}[H]
 \centering
\includegraphics[width=2.5 in]{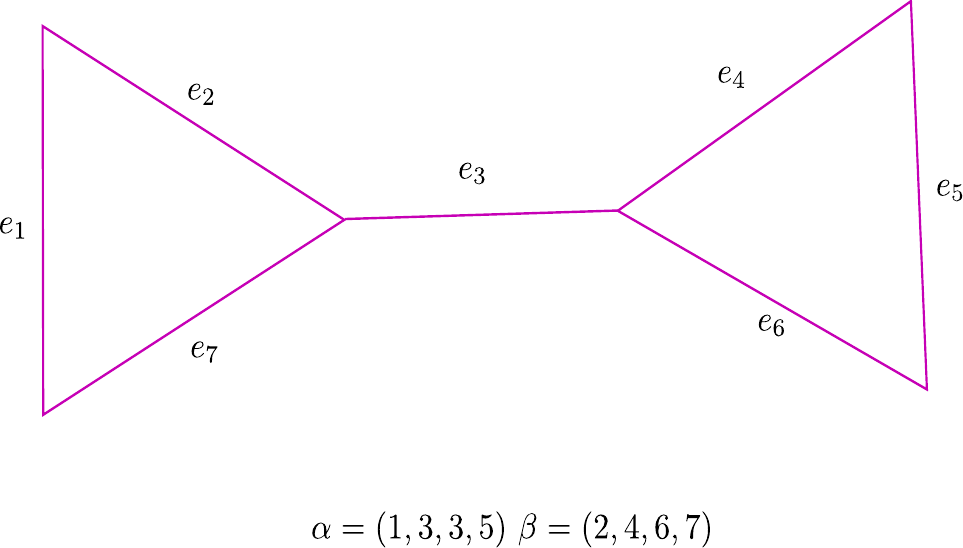} 
\caption{A minimal even walk with repeated facets}\label{fig:bicycle} 
\end{figure}

\end{remark}
\subsection{The structure of even walks}


\begin{prop}[\bf{Structure of even walks}]\label{eqn:luisatheorem}
Let 
$\CC_{\alpha,\beta}=F_{1},F_{2},\dots,F_{2s}$ be an even walk. Then we have 
\renewcommand{\labelenumi}{(\roman{enumi})}
\begin{enumerate}
\item If $i\in \Supp(\alpha)$ (or $i\in\Supp(\beta)$) there exist distinct $j,k\in\Supp(\beta)$ (or $j,k\in\Supp(\alpha)$) such that 
\begin{eqnarray}\label{eqn:intersection}
F_i\cap F_{j}\neq \emptyset&\mbox{and}& F_i\cap F_{k}\neq \emptyset.
\end{eqnarray}
\item The simplicial complex  $\langle\CC_{\alpha,\beta}\rangle$ contains an extended trail of even length labeled $F_{v_1},F_{v_2},\dots,F_{v_{2l}}$ where $v_1,\dots,v_{2l-1}\in\Supp(\alpha)$ and 
$v_2,\dots,v_{2l}\in\Supp(\beta)$.  
\end{enumerate}
\end{prop}
\begin{proof}
\begin{inparaenum}[$(i)$]
To prove \item let $i\in\Supp(\alpha)$, and consider the following set 
\begin{align*}
\A_{i}= \{j\in\Supp(\beta): F_{i}\cap F_{j}\neq \emptyset\}.
\end{align*}
We only need to prove that $|\A_{i}|\geq 2$. 

Suppose $|\A_{i}|=0$ then for all $j\in\Supp(\beta)$ we have 
$$F_i\backslash F_j=F_i\subseteq \{x\in\VV(\CC_{\alpha,\beta}):deg_{\alpha}(x)>deg_{\beta}(x)\}$$
because for each $x\in F_i\backslash F_j$ we have $deg_{\beta}(x)=0$ and $deg_{\alpha}(x)>0$; a contradiction.
 
Suppose $|\A_{i}|= 1$ so that there is one $j\in\Supp(\beta)$ such that $F_i\cap F_{j}\neq \emptyset$. So 
for every $x\in F_i\backslash F_{j}$ we have $deg_{\beta}(x)=0$. Therefore, we have 
$$F_i\backslash F_{j}\subseteq \{x\in \VV(\CC_{\alpha,\beta}):deg_{\alpha}(x)>deg_{\beta}(x)\},$$ 
again a contradiction. So we must have $|\A_{i}|\geq 2$. 

To prove \item pick $u_1\in \Supp(\alpha)$. By using the previous part we can say there are $u_0,u_2\in\Supp(\beta)$, $u_0\neq u_2$, such that 
\begin{eqnarray*}
 F_{u_0}\cap F_{u_1}\neq\emptyset&\mbox{and}&F_{u_1}\cap F_{u_2}\neq\emptyset.
\end{eqnarray*}
By a similar argument there is $u_3\in\Supp(\alpha)$ such that $u_1\neq u_3$ and $F_{u_2}\cap F_{u_3}\neq\emptyset$. 
We continue this process. Pick $u_4\in\Supp(\beta)$ such that 
\begin{eqnarray*}
 F_{u_4}\cap F_{u_3}\neq \emptyset&\mbox{and}& u_4\neq u_2.
\end{eqnarray*}
If $u_4=u_0$, then $F_{u_0},F_{u_1},F_{u_2},F_{u_3}$ is an even length extended trail. If not, we continue this process each time taking   
$$F_{u_0},\dots,F_{u_n}$$
and picking $u_{n+1}\in\Supp(\alpha)$ (or $u_{n+1}\Supp(\beta)$) if $u_n\in\Supp(\beta)$ (or $u_n\in\Supp(\alpha)$) such that 
\begin{eqnarray*}
 F_{u_{n+1}}\cap F_{u_{n}}\neq \emptyset&\mbox{and}&u_{n+1}\neq u_{n-1}.  
\end{eqnarray*}
If $u_{n+1}\in\{u_0,\dots,u_{n-2}\}$, say $u_{n+1}=u_m$, then the process stops and we have 
$$F_{u_m},F_{u_{m+1}},\dots,F_{u_n}$$
is an extended trail. The length of this cycle is even since the indices  $u_m,u_{m+1},\dots,u_n$ 
alternately belong to $\Supp(\alpha)$ and $\Supp(\beta)$ (which are disjoint by our assumption), and if $u_m\in\Supp(\alpha)$, then by construction $u_n\in\Supp(\beta)$ and vice-versa. So there are an even length of such indices  and we are done. 

If $u_{n+1}\notin \{u_0,\dots,u_{n-2}\}$ 
we add it to the end of the sequence and repeat the same process for $F_{u_0},F_{u_1},\dots,F_{u_{n+1}}$. Since $\CC_{\alpha,\beta}$ has a finite number of facets, this process has to stop.    


\end{inparaenum}
\end{proof}
\begin{col}\label{col:new}
An even walk has at least $4$ distinct facets. 
\end{col}
In Corollary~\ref{col:mine}, we will see that every even walk must contain a simplicial cycle.  
\begin{theorem}\label{theorem:goodleaf}
A simplicial forest contains no simplicial even walk. 
\end{theorem}
\begin{proof}
Assume the forest $\Delta$ contains an even walk $\CC_{\alpha,\beta}$ where $\alpha,\beta,\in\II_{s}$ and $s\geq 2$ is an integer. 
Since $\Delta$ is a simplicial forest so is its subcollection $\langle\CC_{\alpha,\beta}\rangle$, so by Theorem~\ref{theorem:herzog} 
$\langle\CC_{\alpha,\beta}\rangle$ contains a good leaf $F_0$. 
So we can consider the following order on the facets $F_0,\dots,F_q$ 
of $\langle\CC_{\alpha,\beta}\rangle$
\begin{eqnarray}\label{eqn:goodleaforder}
F_{q}\cap F_0\subseteq\dots \subseteq F_{2}\cap F_0\subseteq F_{1}\cap F_0.
\end{eqnarray}
Without loss of generality we suppose $0\in\Supp(\alpha)$. Since $\Supp(\beta)\neq \emptyset$, we can pick $j\in\{1,\dots,q\}$ to be the smallest index with $F_j\in\Supp(\beta)$.
Now if $x\in F_0\backslash F_j$, by (\ref{eqn:goodleaforder}) we will have $deg_{\alpha}(x)\geq 1$ and $deg_{\beta}(x)=0$, which shows that 
$$F_0\backslash F_j\subset\{x\in V(\CC_{\alpha,\beta});deg_{\alpha}(x)>deg_{\beta}(x)\},$$
a contradiction.

\end{proof}
\begin{col}\label{col:mine}
 Every simplicial even walk contains a simplicial cycle.  
\end{col}

An even walk  is not necessarily an extended trail. For instance see the following example.
\begin{example}\label{ex,counterexample}
Let $\alpha=(1,3,5,7),\beta=(2,4,6,8)$ and $\CC_{\alpha,\beta}=F_1,\dots,F_8$ as in Figure~\ref{fig:newpicture}. It can easily be seen  that $\CC_{\alpha,\beta}$ is an even walk of distinct facets but 
$\CC_{\alpha,\beta}$ is not an extended trail.  
\begin{figure}[H] 
\centering
\includegraphics[width=2 in]{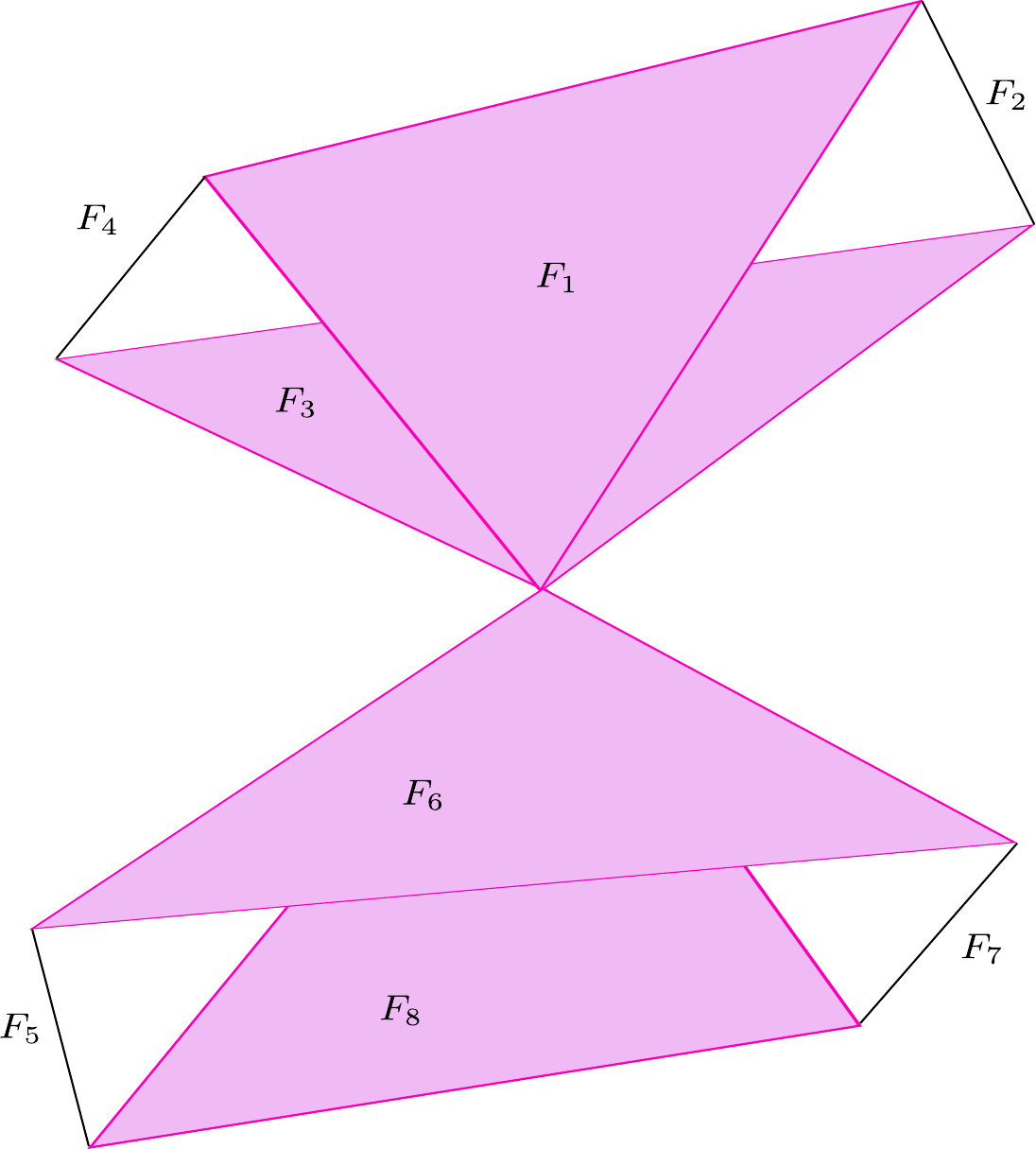}
\caption{An even walk which is not an extended trail.}\label{fig:newpicture}
\end{figure}
The main point here is that we do not require that $F_{i}\cap F_{i+1}\neq \emptyset$ in an even walk which is necessary condition for 
extended trails. For example $F_4\cap F_5\neq \emptyset$ in this case.  
\end{example}
On the other hand, every even-length special cycle is an even walk.   
 \begin{prop}[\bf{Even special cycles are even walks}] \label{eqn:special cycle}
If $F_1,\dots, F_{2s}$ is a special cycle (under the written order) then it is an even walk under the same order. 
 \end{prop}
\begin{proof}
Let $\alpha=(1,3,\dots,2s-1)$ and $\beta=(2,4,\dots,2s)$, and set $\CC_{\alpha,\beta}=F_1,\dots,F_{2s}$. Suppose $\CC_{\alpha,\beta}$ is not an even walk,
so there is $i\in\Supp(\alpha)$ and $j\in \Supp(\beta)$ such that at least one of the following conditions holds
\begin{eqnarray}\label{eqn:one}
F_i\backslash F_j\subseteq\{x\in \VV(\CC_{\alpha,\beta}):deg_{\alpha}(x)>deg_{\beta}(x)\}\\
\nonumber
F_j\backslash F_i\subseteq\{x\in \VV(\CC_{\alpha,\beta}):deg_{\alpha}(x)<deg_{\beta}(x)\}. 
\end{eqnarray}
Without loss of generality we can assume that the first condition holds. 
Pick $h\in\{i-1,i+1\}$ such that $h\neq j$. Then by definition of special cycle there is a vertex $z\in F_i\cap F_{h}$ and $z\notin F_{l}$ for $l\notin \{i,h\}$. In particular, 
$z\in F_i\backslash F_j$, but $deg_{\alpha}(z)=deg_{\beta}(z)=1$ which contradicts (\ref{eqn:one}).  
 
\end{proof}
The converse of Proposition~\ref{eqn:special cycle} is not true: not every even walk is a special cycle, see for example Figure [\ref{figure2}] or Figure [\ref{fig:newpicture}] which are not 
even extended trails. 
But one can show that it is true for even walks with four facets (see~\cite{Alilooee2014}).
\subsection{The case of graphs}

We demonstrate that Definition~\ref{def:scew} in dimension $1$ restricts to closed even walks in graph theory. For more details on the graph theory mentioned in 
this section we refer the reader to \cite{west2001}.
\begin{defn}\label{def:graph}
Let $G=(\VV,E)$ be a graph (not necessarily simple) where $\VV$ is a nonempty set of vertices and $E$ is a set of edges. 
A {\bf walk} of length $n$ in $G$ is a list $e_{1},e_{2},\dots,e_{n}$ of not necessarily distinct edges such that 
$$e_{i}=\{x_{i},x_{i+1}\}\in E\hspace{.4 in} \mbox{for each $i\in \{1,\dots,n-1\}$}.$$
A walk is called {\bf closed} if its endpoints are the same i.e. $x_{1}=x_{n}$. The length of a walk $\mathcal{W}$ is denoted by 
$\ell(\mathcal{W})$. A walk with no repeated edges is called a {\bf trail} and a walk with no repeated vertices or edges is called a {\bf path}.  
A closed walk with no repeated vertices or edges allowed, other than the repetition of the starting and ending vertex, is called a {\bf cycle}.  

\end{defn} 

\begin{lem}[Lemma $1.2.15$ and Remark~$1.2.16$~\cite{west2001}]\label{lem:west}
 Let $G$ be a simple graph. Then we have 
\begin{itemize}
\item Every closed odd walk contains a cycle.
\item Every closed even walk which has at least one non-repeated edge contains a cycle.
\end{itemize} 
\end{lem}

Note that in the graph case the special and simplicial cycles are the ordinary cycles. But extended trails in our definition are not necessarily
cycles in the case of graphs or even a trail. For instance the graph in Figure~\ref{fig:bergecycle} is an extended trail, which is not neither a cycle nor a trail, 
but contains one cycle. This is the case in general.
 \begin{figure}[H]
 \centering
\includegraphics[width=1.5 in]{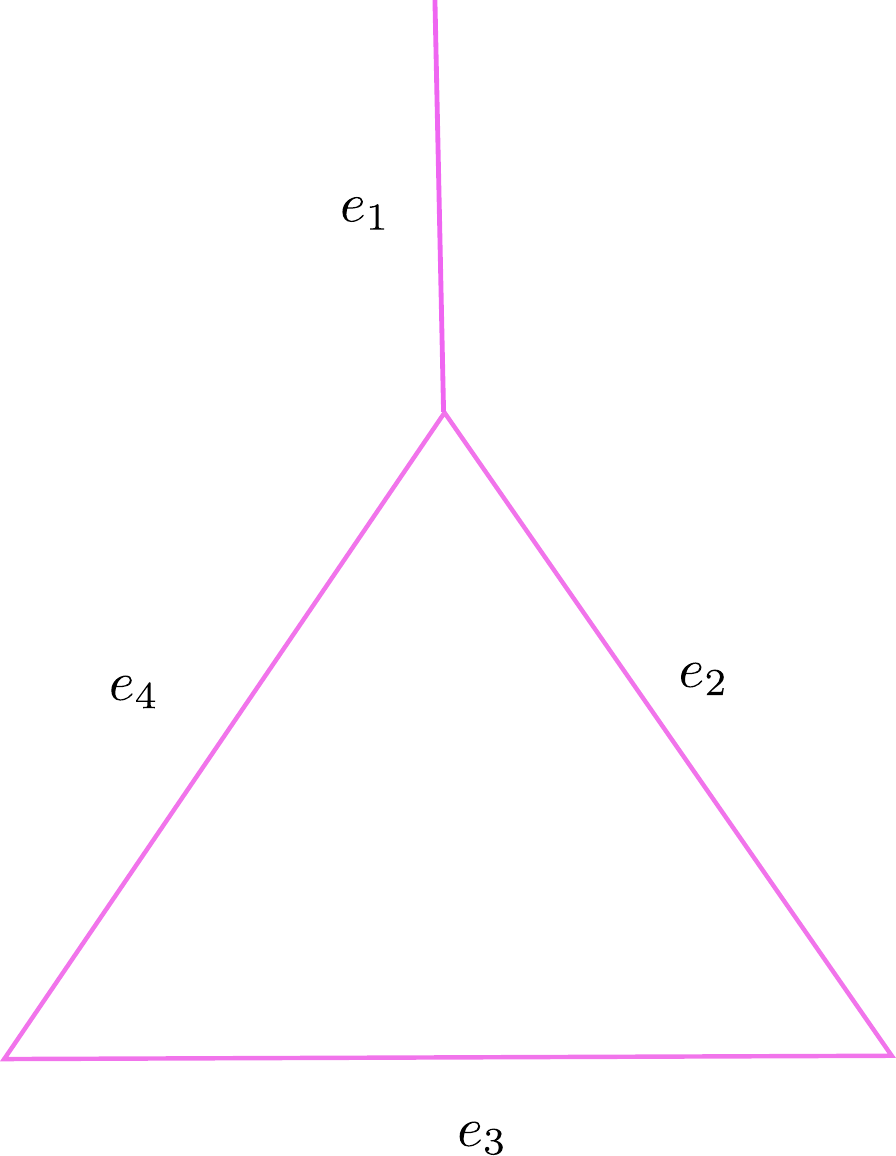} 
\caption{}\label{fig:bergecycle}
\end{figure}

\begin{theorem}[\bf{Euler's Theorem}, \cite{west2001}]\label{thm:euler}
If $G$ is a connected graph, then $G$ is a closed walk with no repeated edges if and only if the degree of every vertex of $G$ is even.
\end{theorem}

\begin{lem}\label{lem:newlem}
Let $G$ be a simple graph and let  $\CC=e_{i_1},\dots,e_{i_{2s}}$ be a sequence of not necessarily distinct edges of $G$ where $s\geq 2$ and  $e_{i}=\{x_{i},x_{i+1}\}$ and $f_i=x_ix_{i+1}$ 
for $1\leq i\leq 2s$.  
Let $\alpha=(i_1,i_3,\dots,i_{2s-1})$ and $\beta=(i_2,i_4,\dots,i_{2s})$. Then 
$\CC$ is a closed even walk if and only if $f_{\alpha}=f_{\beta}$.  
\end{lem}
\begin{proof}
$(\Longrightarrow)$ This direction is clear from the definition of closed even walks. 

$(\Longleftarrow)$ We can give to each repeated edge in $\CC$ a new label and consider $\CC$ as a multigraph (a graph with multiple edges). The condition $f_{\alpha}=f_{\beta}$ implies that 
every $x\in\VV(\CC)$ has even degree, as a vertex of the multigrpah $\CC$ (a graph containing edges that are incident to the same two vertices). Theorem~\ref{thm:euler} implies that $\CC$ is a closed even walk with no repeated edges. Now we revert back to the 
original labeling of  the edges of $\CC$ (so that repeated edges appear again) and then since $\CC$ has even length we are done.     
\end{proof}

To prove the main theorem of this section (Theorem~\ref{them:graphwalk}) we need the following lemma.
\begin{lem}\label{lem,mynewlemma}
Let $\CC=\CC_{\alpha,\beta}$ be a $1$-dimensional simplicial even walk and $\alpha,\beta\in\II_{s}$. If there is $x\in \VV(\CC)$ for which $deg_{\beta}(x)=0$ 
(or $deg_{\alpha}(x)=0$), then we have $deg_{\beta}(v)=0$ ($deg_{\alpha}(v)=0$)
for all $v\in \VV(\CC)$.    
\end{lem}
\begin{proof}
 First we show the following statement.
\begin{eqnarray}\label{stat:new}
e_{i}=\{w_i,w_{i+1}\}\in E(\CC) &\mbox{and}& deg_{\beta}(w_i)=0\Longrightarrow deg_{\beta}(w_{i+1})=0. 
\end{eqnarray}
where $E(\CC)$ is the edge set of $\CC$.

Suppose $deg_{\beta}(w_{i+1})\neq 0$. Then there is $e_{j}\in E(\CC)$ such that $j\in\Supp(\beta)$ and $w_{i+1}\in e_{j}$. On the other hand since $w_i\in e_i$ and $deg_{\beta}(w_i)=0$ we can 
conclude $i\in\Supp(\alpha)$ and thus $deg_{\alpha}(w_i)> 0$. Therefore, we have 
$$e_i\backslash e_j=\{w_i\}\subseteq \{z:deg_{\alpha}(z)>deg_{\beta}(z)\}$$
and it is a contradiction. So we must have $deg_{\beta}(w_{i+1})=0$.

Now we proceed to the proof of our statement. 
Pick  $y\in \VV(\CC)$ such that $y\neq x$. Since $\CC$ is connected we can conclude there is a path $\gamma=e_{i_1},\dots,e_{i_t}$ in $\CC$ in which we have
\begin{itemize}
\item $e_{i_j}=\{x_{i_j},x_{i_{j+1}}\}$ for $j=1,\dots,t$;
\item $x_{i_1}=x$ and $x_{i_{t+1}}=y$.
\end{itemize}
Since $\gamma$ is a path it has neither repeated vertices nor repeated edges. Now note that since $deg_{\beta}(x)=deg_{\beta}(x_{i_1})=0$ 
and $\{x_{i_1},x_{i_2}\}\in E(\CC)$ from (\ref{stat:new}) we have $deg_{\beta}(x_{i_2})=0$. By repeating a similar argument we have 
\begin{eqnarray*}
deg_{\beta}(x_{i_j})=0&\mbox{for $j=1,2,\dots,t+1$.}
\end{eqnarray*}
In particular we have $deg_{\beta}(x_{i_{t+1}})=deg_{\beta}(y)=0$ and we are done. 
\end{proof}
We now show that a simplicial even walk in a graph (considering a graph as a $1$-dimensional simplicial complex) 
 is a closed even walk in that graph as defined in Definition~\ref{def:graph}. 
\begin{theorem}\label{col:newcol}
Let $G$ be a simple graph with edges $e_1,\dots,e_q$. Let 
$e_{i_1},\dots,e_{i_{2s}}$ be a sequence of edges of $G$ such that $\left\langle e_{i_1},\dots,e_{i_{2s}}\right\rangle$ is a connected subgraph of 
$G$ and $\{i_1,i_3,\dots,i_{2s-1}\}\cap\{i_2,i_4,\dots,i_{2s}\}=\emptyset$. Then  
$e_{i_1},\dots,e_{i_{2s}}$ is a simplicial even walk 
if and only if 
$$\{x\in \VV(\CC_{\alpha,\beta}):   deg_{\alpha}(x) > deg_{\beta}(x)\}=\{x\in \VV(\CC_{\alpha,\beta}):   deg_{\alpha}(x) < deg_{\beta}(x)\}=\emptyset$$ 
\end{theorem}
\begin{proof}
($\Longleftarrow$) is clear. To prove the converse we assume $\alpha=(i_1,i_3,\dots,i_{2s-1})$, $\beta=(i_2,i_4,\dots,i_{2s})$ \
and $\CC_{\alpha,\beta}$ is a simplicial even walk. We only need to show   
\begin{eqnarray*}
deg_{\alpha}(x)= deg_{\beta}(x)\hspace{.3 in}\mbox{for all $x\in \VV(\CC_{\alpha,\beta})$}.
\end{eqnarray*}       

Assume without loss of generality $deg_{\alpha}(x)>deg_{\beta}(x)\geq 0$, so there exists $i\in \Supp(\alpha)$ such that $x\in e_i$. We set 
$e_i=\{x,w_1\}$. 

Suppose $deg_{\beta}(x)\neq 0$. 
We can choose an edge  $e_{k}$ in $\CC_{\alpha,\beta}$ where $k\in \Supp(\beta)$
such that $x\in e_{i}\cap e_{k}$. We consider two cases. 
\renewcommand{\labelenumi}{(\arabic{enumi})}
\begin{enumerate}
\item If $deg_{\beta}(w_1)=0$, then since $deg_{\alpha}(w_1)\geq 1$ we have 
$$e_{i}\backslash e_{k}=\{w_1\}\subseteq \{z\in \VV(G):   deg_{\alpha}(z) > deg_{\beta}(z)\},$$ 
a contradiction. 
\item If $deg_{\beta}(w_1)\geq 1$, then there exists $h\in \Supp(\beta)$ with $w_1\in e_{h}$. So we have 
$$e_{i}\backslash e_{h}=\{x\}\subseteq \{z\in \VV(G):   deg_{\alpha}(z) > deg_{\beta}(z)\},$$
again a contradiction. 
\end{enumerate}
So we must have $deg_{\beta}(x)=0$. By Lemma~\ref{lem,mynewlemma} this implies that $deg_{\beta}(v)=0$ for every $v\in\VV(\CC_{\alpha,\beta})$, a contradiction, since 
$\Supp(\beta)\neq\emptyset$. 
\end{proof}

\begin{col}[\bf{$1$-dimensional simplicial even walks}]\label{them:graphwalk}
Let $G$ be a simple graph with edges $e_1,\dots,e_q$. Let 
$e_{i_1},\dots,e_{i_{2s}}$ be a sequence of edges of $G$ such that $\left\langle e_{i_1},\dots,e_{i_{2s}}\right\rangle$ is a 
connected subgraph of $G$ and $\{i_1,i_3,\dots,i_{2s-1}\}\cap\{i_2,i_4,\dots,i_{2s}\}=\emptyset$. Then  
$e_{i_1},\dots,e_{i_{2s}}$ is a simplicial even walk 
if and only if $e_{i_1},\dots,e_{i_{2s}}$ is a closed even walk in $G$.  
\end{col}
\begin{proof}
 Let $I(G)=(f_1,\dots,f_q)$ be the edge ideal of $G$ and $\alpha=(i_1,i_3,\dots,i_{2s-1})$ and $\beta=(i_2,i_4,\dots,i_{2s})$ so that 
$\CC_{\alpha,\beta}=e_{i_1},\dots,e_{i_{2s}}$. Assume $\CC_{\alpha,\beta}$ is a closed even walk in $G$. 
Then we have 
\begin{align*}
f_{\alpha}=\prod_{x\in \VV(\CC_{\alpha,\beta})}x^{deg_{\alpha}(x)}=\prod_{x\in \VV(\CC_{\alpha,\beta})}x^{deg_{\beta}(x)}=f_{\beta}, 
\end{align*}  
where the second equality follows from Lemma~\ref{lem:newlem}.

So for every $x\in \VV(\CC_{\alpha,\beta})$ we have $deg_{\alpha}(x)=deg_{\beta}(x)$. In other words we have 
$$\{x\in \VV(\CC_{\alpha,\beta}):   deg_{\alpha}(x) > deg_{\beta}(x)\}=\{x\in \VV(\CC_{\alpha,\beta}):   deg_{\alpha}(x) < deg_{\beta}(x)\}=\emptyset,$$
and therefore we can say $\CC_{\alpha,\beta}$ is a simplicial even walk.
The converse follows directly from Theorem~\ref{col:newcol} and Lemma~\ref{lem:newlem}.
\end{proof}
We need the following proposition in the next sections. 
\begin{prop}\label{prop:new}
 Let $\CC_{\alpha,\beta}$ be a $1$-dimensional even walk, and $\langle \CC_{\alpha,\beta}\rangle=G$. Then every vertex of
$G$ has degree $>1$. In particular, $G$ is either an even cycle or contains 
at least two cycles.  \end{prop}
\begin{proof}
Suppose $G$ contains a vertex $v$ of degree $1$. Without 
loss of generality we can assume $v\in e_{i}$ where $i\in\Supp(\alpha)$. So $deg_{\alpha}(v)=1$ and from Theorem~\ref{col:newcol}
we have $deg_{\beta}(v)=1$.  Therefore, there is $j\in\Supp(\beta)$ such that $v\in e_j$. Since $deg(v)=1$ we must have $i=j$, a contradiction since 
$\Supp(\alpha)$ and $\Supp(\beta)$ are disjoint.    

Note that by Corollary~\ref{col:mine} $G$ contains a cycle. Now we show that $G$ contains at least two distinct cycles or it is an even cycle. 

Suppose $G$ contains
only one cycle $C_n$. Then removing the edges of $C_n$ 
leaves a forest of $n$ components. Since every vertex of $G$ has degree $>1$,  
each of the components must be singleton graphs (a null graph with only one vertex). So $G=C_n$. Therefore,  by  Corollary~\ref{them:graphwalk} and the fact 
that $\Supp(\alpha)$ and $\Supp(\beta)$ are disjoint, $n$ must be even.    
\end{proof}

\section{A necessary condition for a squarefree monomial ideal to be of linear type}

We are ready to state one of the main results of this paper which is a combinatorial method to detect irredundant Rees equations of squarefree monomial ideals.
We first show that these Rees equations come from even walks.
\begin{lem}\label{eqn:computationallemma}
Let $I=(f_1,\dots,f_q)$ be a squarefree monomial ideal in the polynomial ring $R$. Suppose  $s,t,h$ are integers with $s\geq 2$, $1\leq h\leq  q$ and $1\leq t \leq s$. 
Let $0\neq\gamma\in R$, $\alpha=(i_1,\dots,i_s),\beta=(j_1,\dots,j_s)\in \II_s$. Then 
\renewcommand{\labelenumi}{(\roman{enumi})}
\begin{enumerate}
 \item $\operatorname{lcm}(f_{\alpha},f_{\beta})=\gamma f_{h}\widehat{f}_{\alpha_{t}}\Longleftrightarrow  
T_{\alpha,\beta}=\lambda\widehat{T}_{\alpha_t} T_{(i_t),(h)}+\mu T_{\alpha_{t}(h),\beta}$ for some monomials $\lambda,\mu\in R$, $\lambda\neq 0$. 
\item $\operatorname{lcm}(f_{\alpha},f_{\beta})=\gamma f_{h}\widehat{f}_{\beta_{t}}\Longleftrightarrow 
T_{\alpha,\beta}=\lambda\widehat{T}_{\beta_t} T_{(h),(j_t)}+\mu T_{\alpha,\beta_t(h)}$ for some monomials $\lambda,\mu\in R$, $\lambda\neq 0$. 
\end{enumerate}
\end{lem}
\begin{proof}
We only prove $(i)$; the proof of $(ii)$ is similar.

First note that if $h=i_t$ then $(i)$ becomes 
$$\operatorname{lcm}(f_{\alpha},f_{\beta})=\gamma f_{\alpha}\Longleftrightarrow  
T_{\alpha,\beta}=T_{\alpha,\beta}\hspace{.2 in}\text{(Setting $\mu=1$)}$$
and we have nothing to prove, so we assume that $h\neq i_t$. 

If we have $\operatorname{lcm}(f_{\alpha},f_{\beta})=\gamma f_{h}\widehat{f}_{\alpha_{t}}$, then the monomial $\gamma f_{h}$ is divisible by $f_{i_t}$, 
so there exists a nonzero exists a monomial 
$\lambda\in  R$ such that 
\begin{eqnarray}\label{eqn:equation4}
\lambda\operatorname{lcm}(f_{i_t},f_{h})=\gamma f_{h}.
\end{eqnarray}
It follows that 
\begin{eqnarray}\label{eqn:T}
\nonumber
\displaystyle T_{\alpha,\beta}&=&\displaystyle\left(\frac{\operatorname{lcm}(f_{\alpha},f_{\beta})}{f_{\alpha}}\right)T_{\alpha}-
\left(\frac{\operatorname{lcm}(f_{\alpha},f_{\beta})}{f_{\beta}}\right)T_{\beta}=
\displaystyle \left(\frac{\gamma f_h}{f_{i_t}}\right)T_{\alpha}-\left(\frac{\operatorname{lcm}(f_{\alpha},f_{\beta})}{f_{\beta}}\right)T_{\beta}\\
\nonumber
\\
T_{\alpha,\beta}&=&\displaystyle\lambda\widehat{T}_{\alpha_t}T_{(i_t),(h)}+\left(\frac{\lambda\operatorname{lcm}(f_{i_t},f_{h})}{f_{h}}\right)
{T}_{\alpha_t(h)}-\left(\frac{\operatorname{lcm}(f_{\alpha},f_{\beta})}{f_{\beta}}\right)T_{\beta}. 
\end{eqnarray}
On the other hand note that since we have 
\begin{eqnarray}\label{eqn:equation5}
\operatorname{lcm}(f_{\alpha},f_{\beta})=\gamma f_{h}\widehat{f}_{\alpha_{t}}=\gamma f_{\alpha_t(h)},
\end{eqnarray}
we see $\lcm(f_{\alpha_{t}(h)},f_{\beta})$ divides $\lcm(f_{\alpha},f_{\beta})$. Thus there exists a monomial $\mu\in R$ such that 
\begin{equation}\label{eqn:13}
\lcm(f_{\alpha},f_{\beta})=\mu\lcm(f_{\alpha_{t}(h)},f_{\beta}). 
\end{equation}
By (\ref{eqn:equation4}), (\ref{eqn:equation5}) and (\ref{eqn:13}) we have 
\begin{eqnarray}
\frac{\lambda\operatorname{lcm}(f_{i_t},f_{h})}{f_{h}}=\frac{\lambda\operatorname{lcm}(f_{i_t},f_{h})\widehat{f}_{\alpha_t}}{f_{\alpha_{t}(h)}}=
\frac{\gamma f_h\widehat{f}_{\alpha_t}}{f_{\alpha_{t}(h)}}=
\frac{\lcm(f_{\alpha},f_{\beta})}{f_{\alpha_{t}(h)}}=
\frac{\mu\lcm(f_{\alpha_{t}(h)},f_{\beta})}{f_{\alpha_{t}(h)}}.\label{eqn:newsara}
\end{eqnarray}
Substituting  (\ref{eqn:13}) and (\ref{eqn:newsara}) in (\ref{eqn:T}) we get
$$T_{\alpha,\beta}=\lambda\widehat{T}_{\alpha_t}T_{(i_t),(h)}+\mu T_{\alpha_{t}(h),\beta}.$$
For the converse since $h\neq i_t$, by comparing coefficients we have  
\begin{eqnarray*}
\frac{\operatorname{lcm}{(f_{\alpha},f_{\beta})}}{f_{\alpha}}=\lambda\left(\frac{\lcm(f_{i_t},f_h)}{f_{i_t}} \right)=\lambda \prod_{x\in F_{h}\backslash F_{i_t}}x&\Longrightarrow&
\operatorname{lcm}{(f_{\alpha},f_{\beta})}=\lambda\left(\prod_{x\in F_{h}\backslash F_{i_t}}x\right)f_{\alpha}\\
&\Longrightarrow& \operatorname{lcm}{(f_{\alpha},f_{\beta})}=\lambda_0 f_{h}\hat{f}_{\alpha_t}
\end{eqnarray*}
where $0\neq \lambda_0\in R$. This concludes our proof. 
\end{proof}
Now we show that there is a direct connection between  redundant Rees equations and the above lemma. 
\begin{theorem}\label{eqn:maintheorem}
Let $\Delta=\left\langle  F_1,\dots,F_q\right\rangle $ be a simplicial complex
, $\alpha,\beta\in \II_{s}$ and $s\geq 2$ an integer.
If $\CC_{\alpha,\beta}$ 
is not an even walk then 
$$T_{\alpha,\beta}\in J_1S+J_{s-1}S.$$ 
\end{theorem}
\begin{proof}
Let $I=(f_1,\dots,f_q)$  be the facet ideal of $\Delta$ and let $\alpha=(i_1,\dots,i_s),\beta=(j_1,\dots,j_s)\in \II_{s}$. 
 If $C_{\alpha,\beta}$ is not an even walk, then by 
Definition~\ref{def:scew} there exist 
$i_t\in \Supp(\alpha)$ and $j_l\in \Supp(\beta)$ such that one of the following is true
\renewcommand{\labelenumi}{(\arabic{enumi})}
\begin{enumerate}
 \item $F_{j_l}\backslash F_{i_t}\subseteq \{x\in \VV(\Delta):   deg_{\alpha}(x) < deg_{\beta}(x)\}$;
\item $F_{i_t}\backslash F_{j_l}\subseteq \{x\in \VV(\Delta):   deg_{\alpha}(x) > deg_{\beta}(x)\}$.
\end{enumerate}
Suppose $(1)$ is true. Then there exists a monomial $m\in R$ such that 
\begin{eqnarray}\label{eqn:useful}
\frac{\operatorname{lcm}{(f_{\alpha},f_{\beta})}}{f_{\alpha}}=\prod_{deg_{\beta}(x) > deg_{\alpha}(x)} x^{deg_{\beta}(x)-deg_{\alpha}(x)}=m \prod_{x\in F_{j_l}\backslash F_{i_t}}x.
\end{eqnarray}
So we have 
$$\operatorname{lcm}{(f_{\alpha},f_{\beta})}= mf_{\alpha} \prod_{x\in F_{j_l}\backslash F_{i_t}}x =m_0f_{j_l}\widehat{f}_{\alpha_{t}}$$
where $m_0\in R$. On the other hand by Lemma~\ref{eqn:computationallemma} there exist monomials $0\neq\lambda,\mu\in R$ such that 
\begin{eqnarray*}
 T_{\alpha,\beta}&=&\lambda\widehat{T}_{\alpha_t}T_{(i_t),(j_l)}+\mu T_{\alpha_t(j_l),\beta}\\
&=&\lambda\widehat{T}_{\alpha_t}T_{(i_t),(j_l)}+\mu T_{j_l} T_{\widehat{\alpha}_{t},\widehat{\beta}_{l}}\in J_1S+J_{s-1}S\hspace{.2 in}(\mbox{since $j_l\in\Supp(\beta)$}).
\end{eqnarray*}
If case $(2)$ holds, a similar argument settles our claim.

\end{proof}
\begin{col}\label{eqn:maintheorem2}
Let $\Delta=\left\langle  F_1,\dots,F_q\right\rangle $ be a simplicial complex
and $s\geq 2$ be an integer. Then we have 
$$J=J_1S+\left(\bigcup_{i=2}^{\infty}P_{i}\right)S$$ 
where $P_{i}=\{T_{\alpha,\beta}:\alpha,\beta\in\II_{i}\hspace{.05 in}\mbox{ and $\CC_{\alpha,\beta}$ is an even walk}\}$.
\end{col}
\begin{theorem}[{\bf Main Theorem}]\label{col:linear}
 Let $I$ be a squarefree monomial ideal in $R$ and suppose the facet complex $\FFF(I)$ has no even walk. Then $I$ is of linear type.  
\end{theorem}
The following theorem, can also be deduced from combining Theorem 1.14 in \cite{jahan2012} and Theorem 2.4 in \cite{conca1999}.
In our case, it follows directly from Theorem~\ref{col:linear} and   Theorem~\ref{theorem:goodleaf}.
\begin{col}
 The facet ideal of a simplicial forest is of linear type. 
\end{col}
The converse of Theorem~\ref{eqn:maintheorem} is not in general true. For example see the following.
\begin{example}\label{example:goodexample} Let $\alpha=(1,3),\beta=(2,4)$. In Figure [\ref{figure1}] we see $\CC_{\alpha,\beta}=F_1,F_2,F_3,F_4$ is an even walk 
but  
we have 
$$T_{\alpha,\beta}=x_4x_8 T_1T_3-x_1x_6 T_2T_4=x_8T_3(x_4T_1-x_2T_5)+T_5(x_2x_8T_3-x_5x_6T_4)+x_6T_4(x_5T_5-x_1T_2)\in J_1S .$$ 
\begin{figure}[H]
  \centering
\includegraphics[width=1.5 in]{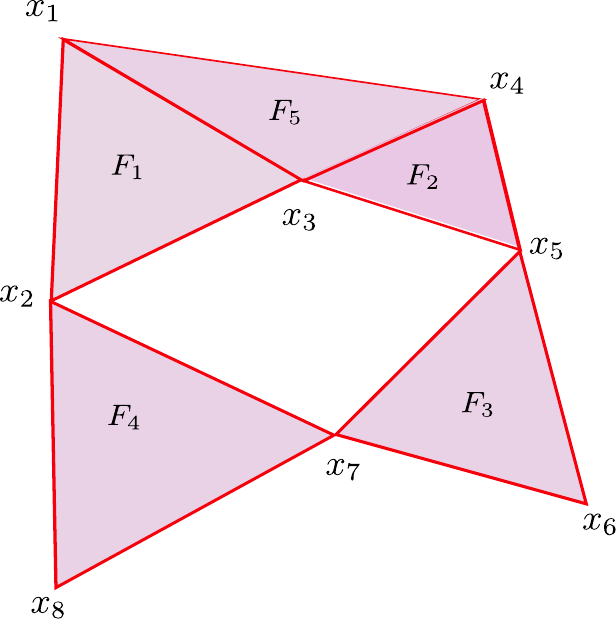}
  \caption{}\label{figure1}
\end{figure}
\end{example}
By Theorem~\ref{eqn:maintheorem}, all irredundant generators of $J$ of deg $>1$ correspond to even walks. However  
irredundant generators of $J$ do not correspond to minimal even walks in $\Delta$ (even walks that do not properly contain other even walks). For instance $\CC_{(1,3,5),(2,4,6)}$ 
as displayed in Figure~[\ref{figure2}] is an even walk which is not minimal (since  $\CC_{(3,5),(2,4)}$ and $\CC_{(1,5),(2,6)}$ are even walks which contain properly in  
$\CC_{(1,3,5),(2,4,6)}$). But $T_{(1,3,5),(2,4,6)}\in J$ is an irredundant generator of $J$.

We can now state a simple necessary condition for a simplicial complex to be of linear type in terms of its line graph.
\begin{defn}
Let $\Delta=\left\langle F_1,\dots, F_n\right\rangle $ be a simplicial complex. The {\bf line graph} $L(\Delta)$ of  $\Delta$ 
is a graph whose vertices are labeled with the facets of $\Delta$, and two vertices labeled $F_i$ and $F_j$ are adjacent if and only if $F_{i}\cap F_j\neq \emptyset$. 
\end{defn}
\begin{theorem}[\bf{A simple test for linear type}]\label{eqn:mycol}
Let $\Delta$ be a simplicial complex and suppose $L(\Delta)$ contains no even cycle. Then $\FFF(\Delta)$ is of linear type.
 \end{theorem}
\begin{proof}
We show that $\Delta$ contains no even walk $\CC_{\alpha,\beta}$. Otherwise by Proposition~\ref{eqn:luisatheorem} 
$\CC_{\alpha,\beta}$ contains an even extended trail $B$, and $L(B)$ is then an even cycle contained in $L(\Delta)$ which is a 
contradiction. Theorem~\ref{col:linear} settles our claim.   
\end{proof}
Theorem~\ref{eqn:mycol} generalizes results of Lin and Fouli~\cite{fouli2013}, where they showed if $L(\Delta)$ is a tree or is an odd cycle then $I$ is of linear type. 

It must be noted, however, that the converse of Theorem~\ref{eqn:mycol} is not true. Below is an example of an ideal of linear type whose line graph contains an even cycle. 
\begin{example}
In the following simplicial complex $\Delta$,  $L(\Delta)$ contains an even cycle but its facet ideal $\FFF(\Delta)$ is of 
linear type.  
\begin{figure}[H]
 \centering
\includegraphics[width=2.5 in]{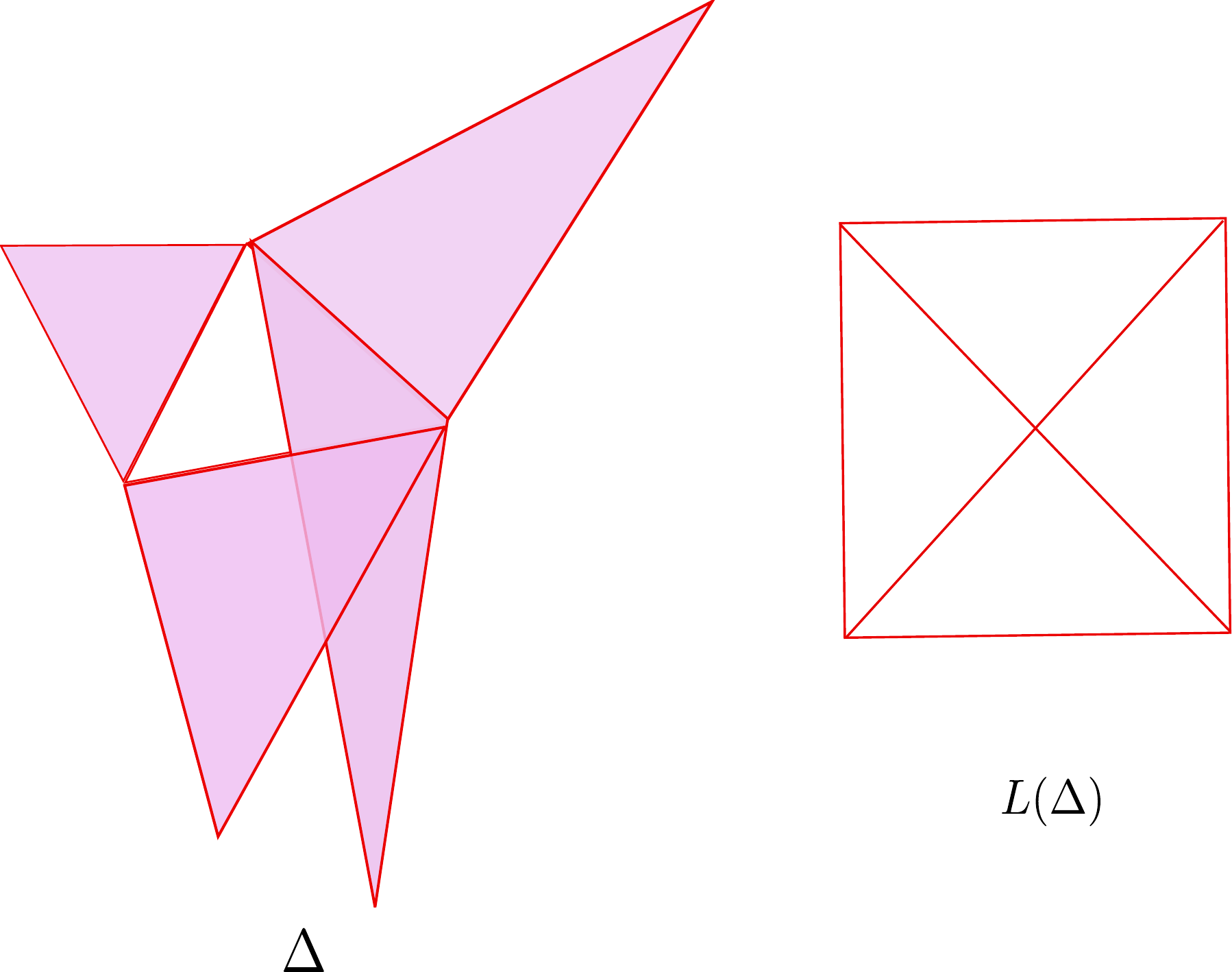}
\end{figure}
\end{example}
By applying Theorem~\ref{col:linear} and Proposition~\ref{prop:new} we conclude the following statement, which was originally proved by Villarreal in \cite{Villarreal1995}.
\begin{col}\label{col:vill}
 Let $G$ be a graph which is either tree or contains a unique cycle and that cycle is odd. Then the edge ideal $\FFF(G)$ is of linear type.
\end{col}



\section*{Acknowledgments}


This paper was prepared while the authors were visiting MSRI. We are grateful to Louiza Fouli, Elizabeth Gross and Elham Roshanbin for further discussions, 
Jonathan Montano and Jack Jeffries for feedback and 
to MSRI for their hospitality. We also would like to thank the anonymous referee for his or her helpful comments. The results of this paper were obtained 
with aid of the the computer algebra software Macualay 2~\cite{Macaulay2}.

\end{document}